\documentclass[11pt,a4paper,reqno]{amsart}
\usepackage[T1]{fontenc}
\usepackage{lmodern}
\usepackage[margin=27mm]{geometry}
\usepackage{amsmath,amssymb,amsthm,mathtools,microtype}
\usepackage[hidelinks]{hyperref}
\hypersetup{pdftitle={Sharp enumeration and growth of admissible sets of integers},
  pdfauthor={Guoqing Wang and Wenkai Yang},
  pdfsubject={Enumeration, infinite admissible sets, and stratification}}
\numberwithin{equation}{section}
\newtheorem{theorem}{Theorem}[section]
\newtheorem{lemma}[theorem]{Lemma}
\newtheorem{proposition}[theorem]{Proposition}
\newtheorem{corollary}[theorem]{Corollary}
\theoremstyle{definition}

\theoremstyle{remark}
\newtheorem{remark}[theorem]{Remark}
\newcommand{\N}{\mathbb N}
\newcommand{\Z}{\mathbb Z}
\newcommand{\R}{\mathbb R}
\newcommand{\FS}{\operatorname{FS}}
\newcommand{\abs}[1]{\lvert #1\rvert}
\newcommand{\q}{\mathfrak q}
\newcommand{\bo}{b_{\mathrm{odd}}}

\title[Sharp enumeration and growth of admissible sets]{Sharp enumeration and growth\\of admissible sets of integers}
\author[Guoqing Wang]{Guoqing Wang}
\address{School of Mathematical Sciences\newline
Tiangong University\newline
Tianjin 300387, P.R. China}
\email{gqwang1979@aliyun.com}
\author[Wenkai Yang]{Wenkai Yang}
\address{Center for Combinatorics\newline
Nankai University\newline
Tianjin 300071, P.R. China}
\email{9820250116@nankai.edu.cn}
\date{}


\begin{document}
\begin{abstract}
A set of positive integers is admissible if equal sums of finite subsets
always have equal numbers of summands. Let $a(N)$ count the nonempty
admissible subsets of $\{1,\ldots,N\}$, and let $b(N)$ count those with
maximum $N$. We prove
$\log a(N)\sim\log b(N)\sim\pi\sqrt{2N/3}$, establishing the sharp
logarithmic asymptotic proposed by Nicolas. More generally, let $Q_h(N)$
count the subsets for which each integer has subset-sum representations
of at most $h$ distinct cardinalities. Uniformly for all positive integers
$h,N$, we prove $\log Q_h(N)\asymp\min\{N,\sqrt{hN}\}$ with absolute
implied constants. Every infinite set satisfying the same condition for
some fixed $h$ has counting function $o(\sqrt x)$.
At $h=1$, the uniform count proves the conjecture
$\log a(N)\asymp\sqrt N$ of Erd\H{o}s, Nicolas and S\'ark\"ozy,
and the infinite-set theorem strengthens their conjectured
vanishing lower limit to a full limit.
The proofs record each subset sum together with its cardinality and use
the counting theorem of Campos and the structure theorem of Carpenter,
Defant and Kravitz. The sharp count also uses Nicolas's partition
enumeration and an elementary construction of a long restricted-sum
progression. Finally, we give an exact partition formula and asymptotics
for the odd sets with separated alternate cardinality layers suggested
by Nicolas, and deduce lower bounds for the number of non-stratified
admissible sets.
\end{abstract}
\maketitle

\section{Introduction}

We use $\N=\{1,2,\ldots\}$, $[N]=\{1,\ldots,N\}$, and natural
logarithms. The sets of integers and real numbers are denoted by $\Z$
and $\R$, respectively.

For a finite set $A$ of real numbers, let
$$
 \FS(A)=\left\{\sum_{a\in U}a:U\subseteq A\right\},\qquad
 j^{\wedge}A=\left\{\sum_{a\in U}a:U\subseteq A,\ \abs U=j\right\}
 \quad(0\le j\le\abs A).
$$
The empty sum is included. A finite or infinite set $A\subseteq\N$ is
\emph{admissible} if
\[
 \sum_{a\in U}a=\sum_{a\in V}a\quad\Longrightarrow\quad\abs U=\abs V
 \qquad\text{for all finite }U,V\subseteq A.
\]
Thus a subset sum determines the number of its summands, although it need
not determine the summands themselves. For instance, $\{4,5,6,7\}$ is
admissible and satisfies $4+7=5+6$. For finite $A$, admissibility is
equivalent to disjointness of the sets $j^{\wedge}A$ for different $j$.

For $g>0$, $f=O(g)$ or $f\ll g$ means $|f|\le k g$
eventually for some $k>0$. The notation $f=o(g)$ means $f/g\to0$, and
$f\sim g$ means $f/g\to1$. For positive $f,g$, $f\gg g$
means $g\ll f$, and $f\asymp g$ means $f\ll g\ll f$.
Subscripts allow dependence on fixed parameters: $f(n)=O_c(1)$
means $|f(n)|\le k(c)$ for $n\ge n_0(c)$, independently of
the sets being counted. Uniformity in a parameter means that both
the bound and the threshold are independent of that parameter.
Parameters specified as fixed remain fixed in each limit.

Erd\H{o}s studied this condition in 1962 \cite{Erdos1962}. Straus
\cite{Straus1966} introduced the terminology and investigated the maximum
size $F(N)$ of an admissible subset of $[N]$. The consecutive integers at
the upper end of $[N]$ give examples of size $(2+o(1))\sqrt N$, and
Erd\H{o}s conjectured that such an interval is extremal. Erd\H{o}s, Nicolas
and S\'ark\"ozy \cite{ENS1991} improved Straus's upper bound. Deshouillers
and Freiman subsequently proved $F(N)=(2+o(1))\sqrt N$
\cite{DF1995}, and then established the terminal-interval conjecture for
all sufficiently large $N$ \cite{DF1999}.

The finite extremal results do not determine the number of admissible
sets or whether the counting function of one infinite admissible set
can stay above a positive multiple of $\sqrt x$.
Erd\H{o}s, Nicolas and S\'ark\"ozy formulated conjectures on both
questions in \cite{ENS1991}. Define
\[
\begin{split}
 a(N)&=\#\{\varnothing\ne A\subseteq[N]:A\text{ is admissible}\},\\
 b(N)&=\#\{A\subseteq[N]:A\text{ is admissible},\ \max A=N\},
 \qquad A(x)=\abs{A\cap[1,x]}.
\end{split}
\]
Classification by the largest element gives
$a(N)=\sum_{n=1}^{N}b(n)$.
On printed page~64, they conjectured
\begin{equation}\label{eq:count-conjecture}
 \log a(N)\asymp\sqrt N.
\end{equation}
Their available estimates differed by a factor $\log N$ in the exponent.
On printed page~65, at the start of their discussion of infinite sets,
they conjectured that every infinite admissible $A$ satisfies
\begin{equation}\label{eq:infinite-conjecture}
 \liminf_{x\to\infty}\frac{A(x)}{\sqrt x}=0.
\end{equation}
We prove both assertions, determine the leading constant in
$\log a(N)$, and strengthen \eqref{eq:infinite-conjecture} to a full limit.

The proofs extend to a common relaxation of admissibility. For any
$A\subseteq\N$ and integer $s$, define
$$
 L_A(s)=\left\{\abs U:U\subseteq A\text{ finite},\
                         \sum_{a\in U}a=s\right\},
 \qquad
 \q(A)=\sup_{s\in\Z}\abs{L_A(s)}.
$$
For an infinite set, $\q(A)$ may be infinite. Every individual $L_A(s)$
is finite, since positivity restricts all summands of a representation
of $s$ to $[1,s]$. The empty set has $\q(\varnothing)=1$, and
admissibility means exactly $\q(A)=1$. The condition $\q(A)\le h$ is
inherited by subsets. For positive integers $h,N$, let
\[
\begin{aligned}
 Q_h(N)&=\#\{A\subseteq[N]:\q(A)\le h\},\\
 b_h(N)&=\#\{A\subseteq[N]:\q(A)\le h,\ \max A=N\}.
\end{aligned}
\]
In particular, $Q_1(N)=a(N)+1$ and $b_1(N)=b(N)$.

\begin{theorem}\label{thm:count}
There are absolute constants $c_1,c_2>0$ such that, for every pair of
positive integers $h,N$,
\[
 \exp\!\bigl(c_1\min\{N,\sqrt{hN}\}\bigr)
 \le Q_h(N)
 \le\exp\!\bigl(c_2\min\{N,\sqrt{hN}\}\bigr).
\]
One may take $c_1=(\log2)/4$. If $h\ge N$, then $Q_h(N)=2^N$.
\end{theorem}

The constants in Theorem~\ref{thm:count} do not depend on $h$, so $h$
may grow arbitrarily with $N$. At $h=1$ the estimate proves
\eqref{eq:count-conjecture}. The matching lower bound comes from subsets
of a short interval ending at $N$. For the upper bound, an estimate
$\abs A=O(\sqrt{hN})$ by itself would retain a logarithmic loss. We show
instead that a fixed proportion of $A$ can be chosen with a controlled
pairwise sumset, and then repeat this extraction on the remainder.

A finite set $T\subseteq\N$ is \emph{stratified} if
$$
 \max(j^{\wedge}T)<\min((j+1)^{\wedge}T)
 \qquad(0\le j<\abs T).
$$
Every stratified set is admissible. Let $p(N)$ denote the number of
partitions of $N$ into positive integers. By Nicolas's enumeration
(Theorem~\ref{thm:nicolas-count}), exactly $p(N)$ stratified sets have
maximum $N$. Their admissibility gives $b(N)\ge p(N)$. Summing over the distinct maxima gives
\[
 a(N)\ge\sum_{n=1}^{N}p(n)
 \sim\frac{\sqrt{6N}}{\pi}\,p(N).
\]
The cumulative estimate is justified in Corollary~\ref{cor:cumulative}.

Nicolas proposed the corresponding sharp logarithmic asymptotic for
$b(N)$ on printed page~213 of \cite{Nicolas1999}. His notation $a(N)$
there counts sets with maximum exactly $N$, and therefore corresponds
to our $b(N)$. We prove the following refinement of the case $h=1$ of
Theorem~\ref{thm:count}.

\begin{theorem}\label{thm:sharp}
As $N\to\infty$,
\[
 \log b(N)\sim\log a(N)\sim\log Q_1(N)
 \sim\pi\sqrt{\frac{2N}{3}}.
\]
Equivalently, both $a(N)$ and $b(N)$ equal
$\exp((\pi\sqrt{2/3}+o(1))\sqrt N)$.
\end{theorem}

The stronger relation $b(N)\sim p(N)$, also suggested in
\cite[p.~213]{Nicolas1999}, is not established here.
The leading constant in Theorem~\ref{thm:sharp} concerns $h=1$.
The uniform estimate in Theorem~\ref{thm:count} also permits $h$ to
grow with $N$. For every fixed $t>0$ and all sufficiently large $N$, we show that
deleting at most $t\sqrt N$ elements from any admissible subset of
$[N]$ leaves a set in one residue class with separated cardinality
layers at a suitable distance. These remaining sets have a counting
bound of $N^2p(N)+N+1$, uniformly over the modulus.
The deleted set is still admissible,
and its counting cost vanishes on the $\sqrt N$ scale as
$t\downarrow0$.

\begin{theorem}\label{thm:infinite}
Fix a positive integer $h$. If $A\subseteq\N$ is infinite and
$\q(A)\le h$, then
\[
 A(x)=o(\sqrt x)\qquad(x\to\infty).
\]
\end{theorem}

Taking $h=1$ settles \eqref{eq:infinite-conjecture} in a stronger form.
It also gives a separation between finite extremal sets and the prefixes
of any one infinite admissible set. If $A=\{a_1<a_2<\cdots\}$, then
Theorem~\ref{thm:infinite} implies
\[
 \frac{a_n}{n^2}\longrightarrow\infty,
 \qquad
 \limsup_{n\to\infty}\frac{a_{n+1}-a_n}{n}=\infty.
\]
These consequences are proved in Corollary~\ref{cor:growth}.

Erd\H{o}s, Nicolas and S\'ark\"ozy \cite[Theorem~2]{ENS1991}
constructed an infinite admissible set with $A(x)\gg x^{5-2\sqrt6}$,
and on page~65 asked whether one can obtain
$A(x)\gg x^{1/2-t}$ for arbitrarily small $t>0$.
We do not address this construction problem.

Nicolas also suggested a modular extension of stratification
\cite[p.~206]{Nicolas1999}: take a set of odd integers and require every
$j$-term sum to be smaller than every $(j+2)$-term sum. Parity then
separates the remaining pairs of cardinalities. We give an exact count
for the positive sets in this family. For positive odd $N$, let $\bo(N)$
denote the number of sets $A$ of positive odd integers with $\max A=N$
and
\begin{equation}\label{eq:odd-property}
 \max(j^{\wedge}A)<\min((j+2)^{\wedge}A)
 \qquad(0\le j\le\abs A-2).
\end{equation}
For a singleton the condition is vacuous.

\begin{theorem}\label{thm:odd}
For every positive odd integer $N$,
\begin{equation}\label{eq:odd-formula}
 \bo(N)=p(N+2)-p(N)-1-
                  \sum_{m=1}^{(N-1)/2}p(m).
\end{equation}
Moreover,
\[
 \bo(N)\sim\pi\sqrt{\frac{2}{3N}}\,p(N)
 \qquad(N\to\infty,\ N\text{ odd})
\]
and
\[
 \sum_{\substack{1\le n\le N\\n\text{ odd}}}\bo(n)\sim p(N)
 \qquad(N\to\infty).
\]
\end{theorem}

By Theorem~\ref{thm:odd}, $\bo(N)=o(p(N))$ as $N\to\infty$
through odd integers.
The non-stratified sets counted by $\bo(N)$ give explicit lower bounds for
$b(N)-p(N)$ and $a(N)-\sum_{n=1}^{N}p(n)$ in
Corollary~\ref{cor:odd-excess}.

For sets in $\R$ or $\R^2$, with $u$ in the same space and $v\in\R$,
we use the elementwise operations
\[
 X+Y=\{x+y:x\in X,\ y\in Y\},\qquad
 u+vX=\{u+vx:x\in X\}.
\]
For pairs, addition is coordinatewise. For finite $A\subseteq\Z$,
the proofs also use
\[
  S(A)=\left\{\left(\sum_{a\in U}a,\abs U\right):U\subseteq A\right\}
 \subseteq\Z^2.
\]
If $\q(A)\le h$ and $A\subseteq[N]$, then
$\abs{ S(A)}-1\le h\abs A N$.
For an integer $u$ larger than every element of $A$,
\[
 S(A\cup\{u\})=S(A)\cup\bigl(S(A)+(u,1)\bigr).
\]
The increase in cardinality is
$\abs{(S(A)+(u,1))\setminus S(A)}$.
For points $(s,j)\in S(A)$, the integer $s-uj$ is constant
along each chain in direction $(u,1)$, and its possible values
form the set $-\FS(u-A)$. Thus a small increase produces a
reflected prefix with few subset sums.

For enumeration, disjoint translates of restricted two-term sums give
a large subset with small doubling. We then use Campos's quantitative
counting theorem \cite{Campos2020}. The dependence on the doubling
constant matters because the iteration reaches several cardinality
scales. Green and Morris \cite{GM2016} established earlier counting
results for fixed doubling, and Liu, Mattos and Szab\'o \cite{LMS2025}
subsequently extended the available range. Campos's explicit estimate
suffices for our application.

For sharp enumeration and the infinite-set theorem, we apply the
structural result of
Carpenter, Defant and Kravitz \cite[Theorem~1.2]{CDK2026}.
For fixed $c_0>0$ and a finite $D\subseteq\R_{>0}$ with
$|\FS(D)|\le c_0|D|^2$, this theorem gives $r>0$ and
$D_1\subseteq D\cap r\Z_{>0}$ such that
$|D\setminus D_1|=O_{c_0}(1)$ and
$\sum_{d\in D_1}d/r=O_{c_0}(|D|^2)$. In the finite argument, truncation gives a dense subset of a
short arithmetic progression. Disjoint pairs then produce a long
progression in a fixed cardinality layer, which separates the layers
of the remaining set. For the infinite argument, two elements from
a fixed initial segment bound the progression's step.
The bound on $\sum_{d\in D_1}d/r$ then forces a uniformly
bounded number of elements below an unbounded
sequence of thresholds, which is impossible for an infinite set.

Section~\ref{sec:chains} proves the common reduction.
Section~\ref{sec:enumeration} proves the uniform enumeration, and
Section~\ref{sec:sharp} proves the sharp logarithmic asymptotic.
Sections~\ref{sec:infinite} and~\ref{sec:odd} prove the infinite-set
theorem and the exact odd-family formula, respectively.

\section{Subset sums with recorded cardinalities}\label{sec:chains}

For finite $A\subseteq\N$, the sets at different second coordinates in
$ S(A)$ are disjoint, so
\begin{equation}\label{eq:lifted-layers}
 \abs{ S(A)}=\sum_{j=0}^{\abs A}\abs{j^{\wedge}A}.
\end{equation}
Positivity gives exactly one point with first coordinate zero, namely
$(0,0)$. If $\q(A)\le h$, then each positive integer
$s\le\sum_{a\in A}a$ is the first coordinate of at most $h$ points
of $ S(A)$. Therefore
\begin{equation}\label{eq:lifted-upper}
 \abs{ S(A)}-1\le h\sum_{a\in A}a.
\end{equation}

\begin{lemma}\label{lem:size}
If $A\subseteq[N]$ has $k\ge1$ elements and $\q(A)\le h$, then
\begin{equation}\label{eq:size}
 k^2+5\le6hN.
\end{equation}
\end{lemma}

\begin{proof}
We bound each cardinality layer before summing over the layers.
For every $k$-element set of integers and $0\le j\le k$,
\begin{equation}\label{eq:layer-bound}
 \abs{j^{\wedge}A}\ge j(k-j)+1.
\end{equation}
To verify this, order the elements of $A$ increasingly and start with
the $j$ smallest indices. Increase the largest selected index one step
at a time until it is $k$, then increase the next largest until it is
$k-1$, and continue. There are $j(k-j)$ steps, and every step strictly
increases the corresponding sum. The cases $j=0,k$ are immediate.
Equations~\eqref{eq:lifted-layers} and~\eqref{eq:layer-bound} give
$$
 \abs{ S(A)}-1
 \ge\sum_{j=0}^k\bigl(j(k-j)+1\bigr)-1
 =\frac{k^3+5k}{6}.
$$
By \eqref{eq:lifted-upper} and $A\subseteq[N]$,
$\abs{S(A)}-1\le h\sum_{a\in A}a\le hkN$.
Dividing by $k$ proves \eqref{eq:size}.
\end{proof}

Lemma~\ref{lem:chains} applies without a bound on $\q$.

\begin{lemma}\label{lem:chains}
Let $P\subseteq\N$ be finite and nonempty, and let $u>\max P$ be an
integer. Then
\begin{equation}\label{eq:chain-lemma}
 \abs{\FS(u-P)}
 \le\abs{( S(P)+(u,1))\setminus S(P)}.
\end{equation}
\end{lemma}

\begin{proof}
To count the points of $( S(P)+(u,1))\setminus S(P)$, join
$v$ to $v+(u,1)$ whenever both points belong to $ S(P)$.
Each vertex has at most one predecessor and at most one successor,
and the second coordinate increases along an edge. The components are
therefore finite chains. Translation sends each terminal vertex to a
different point of $( S(P)+(u,1))\setminus S(P)$, and every point of
this difference arises from a terminal vertex. Thus the right-hand side
of \eqref{eq:chain-lemma} counts the chains.

To compare reflected subset sums with these chains, consider the map
$(s,j)\mapsto s-uj$, which is constant on every chain. Its values on $ S(P)$ are exactly
$-\FS(u-P)$: for $U\subseteq P$,
$$
 \sum_{a\in U}a-u\abs U=\sum_{a\in U}(a-u),
$$
and reflection gives a bijection between subsets of $P$ and subsets of
$u-P$. The number of values is at most the number of chains, proving
the assertion.
\end{proof}

If $u\notin P$, separating subsets according to whether they contain
$u$ gives the identity
\[
  S(P\cup\{u\})= S(P)\cup( S(P)+(u,1)).
\]
The two sets on the right need not be disjoint.

\begin{proposition}\label{prop:prefix}
Let $A=\{a_1<\cdots<a_k\}\subseteq[N]$, with $k\ge2$ and
$\q(A)\le h$. There is an index $i$ satisfying
$\lceil k/2\rceil\le i<k$ such that
\[
 D=\{a_{i+1}-a_j:1\le j\le i\}
 \quad\text{satisfies}\quad \abs{\FS(D)}\le3hN.
\]
\end{proposition}

\begin{proof}
We find the required prefix by averaging the translation increments.
Set $A_i=\{a_1,\ldots,a_i\}$, including $A_0=\varnothing$, and let
$$
 z_i=\abs{ S(A_{i+1})}-\abs{ S(A_i)}
     =\abs{( S(A_i)+(a_{i+1},1))\setminus S(A_i)}
 \quad(0\le i<k).
$$
These quantities are nonnegative, and \eqref{eq:lifted-upper} gives
$$
 \sum_{i=0}^{k-1}z_i=\abs{ S(A)}-1\le hkN.
$$
There are $\lfloor k/2\rfloor\ge k/3$ indices in the specified range.
At least one satisfies $z_i\le3hN$. Apply
Lemma~\ref{lem:chains} to $P=A_i$ and $u=a_{i+1}$.
\end{proof}

The following theorem of Carpenter, Defant and Kravitz
\cite[Theorem~1.2]{CDK2026} is used in both the sharp count and the
infinite-set argument. Both bounds are independent of the choice of
the set.

\begin{theorem}\label{thm:CDK}
For every $c_0>0$ there are constants $m_0=m_0(c_0)$ and
$c_1=c_1(c_0)$ such that the following holds. If
$D\subseteq\R_{>0}$ has $i$ elements and
$\abs{\FS(D)}\le c_0i^2$, there is a decomposition
$D=D_1\sqcup D_2$ and a real number $r>0$ satisfying
\begin{equation}\label{eq:CDK}
 \abs{D_2}\le m_0,\qquad
 D_1\subseteq r\Z_{>0},\qquad
 \sum_{d\in D_1}\frac d r\le c_1i^2.
\end{equation}
\end{theorem}

We use the statement in version~2 of \cite{CDK2026}.
For integer inputs, the step can be chosen to be a positive integer.

\begin{corollary}\label{cor:integer-CDK}
If $D\subseteq\N$ satisfies the hypotheses of
Theorem~\ref{thm:CDK}, then its conclusion holds with a positive
integer $r$, using the same constants $m_0,c_1$.
\end{corollary}

\begin{proof}
To obtain an integer step, retain the decomposition and positive
real step $r$ supplied by Theorem~\ref{thm:CDK}.
If $D_1=\varnothing$, all conclusions hold with step $1$.
Otherwise, let $g=\gcd(D_1)$. By B\'ezout's identity there are
integers $n_d$ such that
$g/r=\sum_{d\in D_1}n_d(d/r)\in\Z_{>0}$.
Hence $g\ge r$, $D_1\subseteq g\Z_{>0}$, and
\[
 \sum_{d\in D_1}\frac d g
 \le\sum_{d\in D_1}\frac d r\le c_1i^2.
\]
Replacing $r$ by $g$ proves the assertion.
\end{proof}

Nicolas \cite[Theorem~2]{Nicolas1999} gives the following enumeration
of stratified sets.

\begin{theorem}\label{thm:nicolas-count}
For each positive integer $N$, exactly $p(N)$ stratified sets have maximum $N$.
\end{theorem}

\section{Uniform enumeration}\label{sec:enumeration}

\begin{lemma}\label{lem:translates}
Let $D=\{d_1<\cdots<d_i\}\subseteq\N$, where $i\ge2$.
Let $b=\lceil i/2\rceil$ and $E=\{d_1,\ldots,d_b\}$. Then
\begin{equation}\label{eq:translates}
 \abs{E+E}\le\frac{4\abs{\FS(D)}}{i}+b.
\end{equation}
\end{lemma}

\begin{proof}
We bound the restricted two-term sums by disjoint translates.
If $b=1$, the assertion follows from $\abs{E+E}=1$.
Suppose $b\ge2$, and let $R=2^{\wedge}E$. For
$0\le j\le\lfloor(i-b)/2\rfloor$, we have
\[
 (d_{b+1}+\cdots+d_{b+2j})+R\subseteq\FS(D).
\]
The initial sum uses elements of $D\setminus E$ and equals zero
when $j=0$. Each element of $R$ is the sum of two distinct elements
of $E$.
Moreover, $R\subset(0,2d_b)$, whereas two successive initial sums
differ by more than $2d_b$. Hence these translates are pairwise
disjoint. Their number is
$$
 \left\lfloor\frac{i-b}{2}\right\rfloor+1
 =\left\lfloor\frac i4\right\rfloor+1>\frac i4.
$$
Thus $\abs R\le4\abs{\FS(D)}/i$. Adding the at most $b$ diagonal
sums $2d$ proves \eqref{eq:translates}.
\end{proof}

\begin{proposition}\label{prop:block}
Let $A\subseteq[N]$ have $k\ge4$ elements and $\q(A)\le h$.
There is $B\subseteq A$, of size $b$, such that
\[
 \frac k4\le b\le\frac k2,
 \qquad
 \abs{B+B}\le\frac{27hN}{k}\le\frac{27hN}{b}.
\]
\end{proposition}

\begin{proof}
We obtain the required subset by reflecting the smallest half of $D$.
Choose $i$ and $D$ from Proposition~\ref{prop:prefix}. Let $E$ be
the smallest $b=\lceil i/2\rceil$ elements of $D$. Then
$k/4\le b\le k/2$, and Lemma~\ref{lem:translates} gives
$$
 \abs{E+E}\le\frac{12hN}{i}+b
 \le\frac{24hN}{k}+\frac k2
 \le\frac{27hN}{k},
$$
where the last inequality follows from Lemma~\ref{lem:size}.
Set $B=a_{i+1}-E\subseteq A$. Since
$B+B=2a_{i+1}-(E+E)$, the two sumsets have the same cardinality.
\end{proof}

We apply Theorem~4.1 of the preprint version of Campos's paper
\cite{Campos2020} to subsets of $\Z$. Since the only finite subgroup of $\Z$ is
$\{0\}$, the finite-subgroup parameter in that theorem equals $1$.

\begin{theorem}\label{thm:Campos}
Let $n\ge2$ and $s\ge1$ be integers, and let
$$
 2<K<2^{-36}\frac{s}{(\log n)^3}.
$$
The number of $s$-element sets $J\subseteq[n]$ satisfying
$\abs{J+J}\le Ks$ is at most
\begin{equation}\label{eq:Campos}
 \exp\!\left(2^9\min\left\{\frac K{K-2},\log s\right\}
                 K^{1/6}s^{5/6}\sqrt{\log n}\right)
 \binom{(Ks+1)/2}{s}.
\end{equation}
\end{theorem}

For real $y>s-1$, the binomial coefficient denotes
$y(y-1)\cdots(y-s+1)/s!$.

\begin{lemma}\label{lem:block-count}
There is an absolute integer $N_0$ with the following property.
Suppose $N\ge N_0$ and $1\le h\le N$, and let
$$
 M=hN,\qquad k_0=\lceil M^{2/5}\rceil.
$$
For each integer $b$ in the range
$$
 k_0/4\le b\le\min\{N,\sqrt{6M}\},
$$
the number of $b$-element sets $B\subseteq[N]$ with
$\abs{B+B}\le27M/b$ is at most
\[
 \left(\frac{200M}{b^2}\right)^b.
\]
\end{lemma}

\begin{proof}
The hypotheses of Theorem~\ref{thm:Campos} hold uniformly in $h$
for $(n,s,K)=(N,b,48M/b^2)$. We have $K\ge8$, and the sumset bound under
consideration is at most $Kb$. Uniformly over the stated range,
\[
 \frac{K(\log N)^3}{b}
 =\frac{48M(\log N)^3}{b^3}
 \le\frac{3072(\log N)^3}{M^{1/5}}
 \le\frac{3072(\log N)^3}{N^{1/5}}\longrightarrow0.
\]
An absolute $N_0$ therefore ensures the strict hypothesis in
Theorem~\ref{thm:Campos} for every permitted $h,b$.

We bound the two factors in Campos's estimate.
Since $K\ge8$, we have $K/(K-2)\le4/3$. The exponent in the
exponential factor of \eqref{eq:Campos}, divided by $b$, is at most
$$
 \frac{2^{11}}3
 \left(\frac{48M}{b^3}\right)^{1/6}\sqrt{\log N}
 \le\frac{2^{11}}3\,3072^{1/6}
          \frac{\sqrt{\log N}}{M^{1/30}}
 \le\frac{2^{11}}3\,3072^{1/6}
          \frac{\sqrt{\log N}}{N^{1/30}}\longrightarrow0.
$$
After enlarging the same absolute $N_0$, this factor is at most $e^b$.
Also $M/b\ge1$, so
$$
 \frac{Kb+1}{2}=\frac{24M}{b}+\frac12\le\frac{25M}{b}.
$$
The inequality $\binom yb\le(ey/b)^b$ for $y\ge b$ now yields
$$
 e^b\binom{(Kb+1)/2}{b}
 \le\left(\frac{25e^2M}{b^2}\right)^b
 \le\left(\frac{200M}{b^2}\right)^b.
$$
All thresholds used above are independent of $h$.
\end{proof}

\begin{proof}[Proof of the upper bound in Theorem~\ref{thm:count}]
For $h\ge N$, every subset of $[N]$ satisfies $\q\le h$:
sum zero has only
cardinality zero, and a positive sum has possible cardinalities in
$\{1,\ldots,N\}$. Thus $Q_h(N)=2^N$, as asserted. We may assume
$1\le h\le N$ and define $M=hN$ and $k_0=\lceil M^{2/5}\rceil$.

We encode each remaining set by successive extracted subsets.
Take $N$ sufficiently large for Lemma~\ref{lem:block-count}.
Increase its absolute threshold so that $4\le k_0\le N/2$ uniformly
in $h$, which is possible because
$N^{2/5}\le k_0\le N^{4/5}+1$.
Fix an order on all subsets of $[N]$. Given $A$ with $\q(A)\le h$,
start with $A^{(0)}=A$. While $\abs{A^{(j)}}>k_0$, choose the
first subset $B_j\subseteq A^{(j)}$ satisfying Proposition~\ref{prop:block}
and set $A^{(j+1)}=A^{(j)}\setminus B_j$. Every remainder still has
$\q\le h$, so the procedure applies at each step.
For $b_j=\abs{B_j}$, the extraction gives
\begin{equation}\label{eq:decay}
 b_j\ge\frac{\abs{A^{(j)}}}{4},\qquad
 \abs{A^{(j+1)}}\le\frac34\abs{A^{(j)}},
 \qquad \abs{A^{(j)}}\le(3/4)^j\sqrt{6M}.
\end{equation}
The last inequality uses Lemma~\ref{lem:size}. Stop after $r$ steps
at $R=A^{(r)}$, where $\abs R\le k_0$. Since
$\sqrt{6M}/k_0\le\sqrt6 M^{1/10}$ and $M\le N^2$, the
number of steps satisfies
\begin{equation}\label{eq:steps}
 r\le 1+\frac{\log(\sqrt{6M}/k_0)}{\log(4/3)}
 =O(\log N)
\end{equation}
when $r\ge1$. If $r=0$, the process stops at $A$.

We count the extraction records.
The ordered tuple $(B_0,\ldots,B_{r-1},R)$ determines $A$ by union.
The fixed choice rule therefore gives an injection into such records.
There are at most
\begin{equation}\label{eq:profiles}
 \exp\bigl(O((\log N)^2)\bigr)
 =\exp(o(\sqrt N))
\end{equation}
possible size tuples $(b_0,\ldots,b_{r-1},\abs R)$, using
\eqref{eq:steps} and at most $N+1$ choices for each entry. Every
recorded block has $b_j>k_0/4$ and $b_j\le\min\{N,\sqrt{6M}\}$.
For a fixed size tuple that occurs, Lemma~\ref{lem:block-count} bounds
the number of ordered block choices by
$$
 \prod_{j=0}^{r-1}\left(\frac{200M}{b_j^2}\right)^{b_j}.
$$
Counting these choices independently, without the disjointness or
selection requirements, only increases the bound.

We bound the logarithm of the product by a geometric sum.
Set $f_M(x)=x\log(200M/x^2)$ for $x>0$. It is positive and increasing
on $0<x\le\sqrt{6M}$, since
$$
 f_M'(x)=\log(200M/x^2)-2\ge\log(100/3)-2>0.
$$
Using \eqref{eq:decay}, the logarithm of the preceding product is at
most
\begin{align}
 \sum_{j=0}^{r-1}f_M(b_j)
 &\le\sum_{j=0}^{\infty}f_M((3/4)^j\sqrt{6M})\notag\\
 &=\sqrt{6M}\sum_{j=0}^{\infty}(3/4)^j
       \left(\log\frac{100}{3}+2j\log\frac43\right)\notag\\
 &=\sqrt{6M}\left(4\log\frac{100}{3}+24\log\frac43\right).
 \label{eq:geometric-cost}
\end{align}
The infinite series bounds the finite sum of the values $f_M(b_j)$.
Lemma~\ref{lem:block-count} is used only for the extracted blocks.

Finally, the number of possible terminal remainders is at most
\begin{equation}\label{eq:remainders}
 \sum_{t=0}^{k_0}\binom Nt
 \le(k_0+1)\left(\frac{eN}{k_0}\right)^{k_0}
 =\exp(O(\sqrt M)),
\end{equation}
uniformly in $h$. Since $N\le M\le N^2$ and
$k_0\le2M^{2/5}$ for large $N$, the logarithm of the left-hand
side is at most
\[
 \log(N+1)+2M^{2/5}(1+\log N)=o(\sqrt M)
\]
uniformly in $h$. For $h=1$, we have $k_0=\lceil N^{2/5}\rceil$,
and the logarithm of the number of terminal remainders is
$O(N^{2/5}\log N)=o(\sqrt N)$.
Combining \eqref{eq:profiles}, \eqref{eq:geometric-cost} and
\eqref{eq:remainders} proves $Q_h(N)\le\exp(c_2\sqrt{hN})$ with an
absolute $c_2$ for all sufficiently large $N$ and all $1\le h\le N$.
For $N$ below the fixed threshold, the bound $Q_h(N)\le2^N$ is
absorbed by increasing $c_2$, because
$N/\sqrt{hN}\le\sqrt N$. This completes the upper bound.
\end{proof}

\begin{proof}[Proof of the lower bound in Theorem~\ref{thm:count}]
We construct a terminal interval whose subsets all satisfy the overlap bound.
The case $h\ge N$ follows from $Q_h(N)=2^N$. Suppose $1\le h\le N$,
and let
$$
 \ell=\max\left\{1,\left\lfloor\frac{\sqrt{hN}}2\right\rfloor\right\},
 \qquad I=\{N-\ell+1,\ldots,N\}.
$$
We have $\ell\ge\sqrt{hN}/4$. We claim that $\q(I)\le h$.
For $\ell=1$ this is immediate. If $\ell\ge2$, then $\ell\le N/2$ and
$\ell^2\le hN/4$. For any sum $s\in\FS(I)$, a possible cardinality $j$
belongs to the interval
$$
 \frac{s}{N}\le j\le\frac{s}{N-\ell+1}.
$$
Since $s\le\ell N$, this interval has length at most
$$
 \frac{\ell(\ell-1)}{N-\ell+1}<\frac{2\ell^2}{N}\le\frac h2<h.
$$
An interval of length less than the positive integer $h$ contains at
most $h$ integers. Thus the claim follows, including the sum $s=0$.
Every subset of $I$ has the same bound on $\q$, giving
$$
 Q_h(N)\ge2^\ell\ge\exp\!\left(\frac{\log2}{4}\sqrt{hN}\right).
$$
This proves the lower bound with the stated constant.
\end{proof}

\begin{corollary}\label{cor:maximum}
There are absolute constants $c_1,c_2>0$ such that, for all $h\ge1$ and
$N\ge4$,
$$
 \exp\!\bigl(c_1\min\{N,\sqrt{hN}\}\bigr)
 \le b_h(N)
 \le\exp\!\bigl(c_2\min\{N,\sqrt{hN}\}\bigr).
$$
For $h=1$, one has $b_1(N)=b(N)\ge p(N)$, and
Theorem~\ref{thm:sharp} gives the more precise conclusion
$$
 \lim_{N\to\infty}\frac{\log b_1(N)}{\sqrt N}
 =\pi\sqrt{\frac23}.
$$
\end{corollary}

\begin{proof}
The upper bound is immediate from Theorem~\ref{thm:count}.
For $h\le N$ and $N\ge16$, the interval $I$ used for the lower bound
in Theorem~\ref{thm:count} has $\ell\ge2$ elements.
The $2^{\ell-1}$ subsets of $I$ containing $N$ give
$$
 b_h(N)\ge2^{\ell-1}
 \ge\exp\!\left(\frac{\log2}{8}\sqrt{hN}\right).
$$
For $h\ge N$ the count equals $2^{N-1}$. For the remaining finite
range $4\le N<16$ and $h<N$, decrease the absolute lower constant if
necessary. Every singleton and every two-element positive set is
admissible, so $b_h(N)\ge N>1$.

We identify the sharp constant for $h=1$.
Theorem~\ref{thm:nicolas-count} gives exactly $p(N)$
stratified sets with maximum $N$, each admissible. The
Hardy--Ramanujan asymptotic \cite{HR1918},
\begin{equation}\label{eq:HR}
 p(N)\sim\frac1{4\sqrt3\,N}
              \exp\!\left(\pi\sqrt{\frac{2N}{3}}\right),
\end{equation}
gives the partition lower bound on the logarithmic scale. The exact
limit is the assertion of Theorem~\ref{thm:sharp}, proved in
Section~\ref{sec:sharp}.
\end{proof}

\section{Sharp logarithmic enumeration}\label{sec:sharp}

To determine the leading constant at $h=1$, we first bound the
number of possible deleted sets. We then count the remaining sets
in one residue class with separated cardinality layers by an
injection into the stratified sets.

\begin{lemma}\label{lem:small-admissible-cost}
For each fixed $0<t\le1$, the number of admissible subsets of $[N]$ of size
at most $t\sqrt N$ is at most
\[
 \exp\!\left(\bigl[30t\log(e/t)+o_t(1)\bigr]\sqrt N\right).
\]
\end{lemma}

\begin{proof}
We count the extracted subsets using the common estimate.
Take $h=1$, $M=N$ and $k_0=\lceil N^{2/5}\rceil$ in
Lemma~\ref{lem:block-count}. Uniformly for
$k_0/4\le b\le\sqrt{6N}$, the number of possible $b$-element
subsets with $|B+B|\le27N/b$ is at most $(200N/b^2)^b$.

We encode a small admissible set by successively extracted subsets.
Sets of at most $k_0$ elements are included among the terminal
remainders counted below. Starting with $m>k_0$ elements,
repeatedly apply Proposition~\ref{prop:block}
until at most $k_0$ elements remain, choosing each subset by a fixed
rule. The sizes satisfy $b_j\le(3/4)^jm$ and $b_j>k_0/4$.
There are $O(\log N)$ steps, so the size records have
$\exp(O((\log N)^2))$ possibilities. The terminal remainder has at
most $\exp(O(N^{2/5}\log N))=\exp(o(\sqrt N))$ possibilities.
Since $x\log(200N/x^2)$ is increasing for $x\le\sqrt{6N}$,
the logarithm of the number of choices for the extracted subsets is
at most
\[
 \sum_{j\ge0}(3/4)^jm
 \left(\log\frac{200N}{m^2}+2j\log\frac43\right)
 =4m\log\frac{200N}{m^2}+24m\log\frac43.
\]
The last expression is increasing for $0<m\le\sqrt N$. At
$m=t\sqrt N$ it equals
\[
 t\sqrt N\left(4\log200+24\log\frac43
                    +8\log\frac1t\right)
 \le30t\log(e/t)\sqrt N.
\]
The size records and terminal remainder contribute $o(\sqrt N)$ to
the logarithm of the count. For $m\le k_0$, the whole set is the
terminal remainder.
\end{proof}

\begin{lemma}\label{lem:uniform-modular-cost}
Fix positive integers $r,N$. There are at most $p(N)$ sets
$B\subseteq\N$ with $\max B=N$ such that all elements of $B$
are congruent modulo $r$ and
\[
 \max(j^{\wedge}B)<\min((j+r)^{\wedge}B)
 \qquad(0\le j\le |B|-r).
\]
Consequently the number of such subsets of
$[N]$, allowing the modulus to vary and including the empty set, is
at most $N^2p(N)+N+1$.
\end{lemma}

\begin{proof}
The following injection takes each counted set to a stratified set
with the same maximum.
For fixed $r,N$, the map
\[
 B\longmapsto T=\left\{\frac{b+(r-1)N}{r}:b\in B\right\}
\]
is injective, and $\max T=N$. Since $b\equiv N\pmod r$,
each image element equals $N-(N-b)/r$ and is a positive integer.
For $0\le j<|B|$ with $j+r\le|B|$, extend a subset attaining
$\min((j+1)^{\wedge}B)$ by $r-1$ unused elements. Each added
element is at most $N$, so the assumed separation gives
\[
 \max(j^{\wedge}B)<\min((j+r)^{\wedge}B)
 \le\min((j+1)^{\wedge}B)+(r-1)N.
\]
If $j+r>|B|$, let $h=|B|-j-1\le r-2$.
Complementation and positivity give
\[
 \min((j+1)^{\wedge}B)-\max(j^{\wedge}B)
 =\min((h+1)^{\wedge}B)-\max(h^{\wedge}B)
 >-hN\ge-(r-1)N.
\]
In either case the adjacent-layer difference for $T$ equals
$[\min((j+1)^{\wedge}B)-\max(j^{\wedge}B)+(r-1)N]/r>0$.
Thus $T$ is stratified. Since $\max T=N$, Theorem~\ref{thm:nicolas-count} gives the bound $p(N)$.

We sum this bound over the maximum and modulus.
For sets of at least two elements the modulus divides a nonzero
difference, so it is at most $N$. There are at most $N^2$ pairs of
parameters, and $p(n)\le p(N)$ for $n\le N$.
The singletons and empty set contribute at most $N+1$ further sets.
\end{proof}

\begin{lemma}
\label{lem:sharp-smoothing}
For every fixed $c\in(0,1]$ there is a positive constant $c_1$,
depending only on $c$, with the following property. For all
sufficiently large integers $n$, every $B\subseteq[n]$ with
$|B|\ge cn$ has a disjoint decomposition
$B=B_1\sqcup E\sqcup B'$ such that
\[
 |B_1|\le c_1\log n,\qquad |E|\le c_1,\qquad
 B'\subseteq b+d\mathbb Z
\]
for some integers $b$ and $1\le d\le\lceil4/c\rceil$.
Moreover, for some integer $\ell$, the restricted sumset
$\ell^{\wedge}B_1$ contains an arithmetic progression of common
difference $d$ whose largest element minus its smallest element
is at least $n$.
\end{lemma}

\begin{proof}
Disjoint pairs with one common bounded difference give an initial
progression. Order the elements of $B$ and pair the first with the
second, the third with the fourth, and so on. The sum of these
positive pair differences is less than $n$. Thus fewer than $cn/4$
pairs have difference greater than $4/c$. Since there are at least
$cn/2-1$ pairs altogether, at least $cn/4-1$ pairs have difference
at most $4/c$. One integer $d\le\lceil4/c\rceil$ therefore occurs
as the difference of $\gg_c n$ pairs. For sufficiently large $n$,
select $\ell=\lceil16/c\rceil$ of these pairs, and let $B_1$
consist of their $2\ell$ elements. Choosing one element from each
pair gives an arithmetic progression
$P\subseteq\ell^{\wedge}B_1$ of common difference $d$ and diameter
$\ell d\ge16d/c$.

We enlarge $P$ to diameter at least $(d^2+3)n$.
Allow at most
\[
 \left\lceil16c^{-1}\log n\right\rceil
 +\left\lceil4(d^2+2)/c\right\rceil
\]
further pairs in the two stages below.
Together with the initial pairs, this budget uses
$O_c(\log n)<cn/2$ elements for sufficiently large $n$.
Thus at least $cn/2$ unused elements are available at every
step within this budget.

While the diameter $m=\max P-\min P$ is less than $n$, we
increase it by a factor at least $1+c/8$.
One residue class modulo $d$ contains at least $cn/(2d)$ unused
elements. Partition $[n]$ into consecutive blocks of $m$ integers,
except possibly for a shorter final block. There are at most
$2n/m$ blocks, so one contains at least $cm/(4d)\ge4$ of these
elements. Let $u<v$ be their extreme elements and let $q=v-u$.
Since they belong to one residue class,
\[
 d\mid q,\qquad q\ge cm/4-d\ge cm/8,\qquad q<m,
\]
where the second inequality uses $m\ge16d/c$.
The progressions $P+u$ and $P+v$ overlap, so their union is an
arithmetic progression of common difference $d$ and diameter
$m+q\ge(1+c/8)m$.
Adjoin $u,v$ to $B_1$, replace $P$ by this union, and increase
$\ell$ by one. The inclusion $P\subseteq\ell^{\wedge}B_1$ is
preserved because the two new elements were unused.
Since $\log(1+c/8)\ge c/16$, at most
$\lceil16c^{-1}\log n\rceil$ pairs are needed to reach
diameter at least $n$.

We next increase the diameter to allow for the loss when combining
residue classes. One residue class modulo $d$ still contains
at least $cn/(2d)$ unused elements. Its extreme elements
$u<v$ have difference
\[
 cn/4\le cn/2-d\le v-u<n
\]
for sufficiently large $n$. Since the current diameter is at
least $n$, the union $(P+u)\cup(P+v)$ is again an arithmetic
progression of common difference $d$, with diameter increased
by at least $cn/4$.
Adjoin the endpoints to $B_1$, replace $P$ by this union, and
increase $\ell$ by one.
After $\lceil4(d^2+2)/c\rceil$ such pairs, we have
\[
 P\subseteq\ell^{\wedge}B_1,\qquad
 \max P-\min P\ge(d^2+3)n,\qquad |B_1|=O_c(\log n).
\]

We make the common difference of $P$ compatible with the unused elements.
Discard every residue class modulo $d$ containing fewer than $d^2$
unused elements, and let $E$ consist of the discarded elements.
Then $|E|<d^3$. At least one class remains for large $n$.
Choose a remaining residue $b$. For each other remaining residue $a$,
select $d-1$ pairs $(u_{a,i},v_{a,i})$ from
$B\setminus(B_1\cup E)$ with
$u_{a,i}\equiv b$ and $v_{a,i}\equiv a\pmod d$.
Every remaining class contains at least $d^2$ unused elements.
The choices require at most $(d-1)^2<d^2$ elements from the base
class and $d-1$ from each other class, so all endpoints can be
chosen distinct.

We fill all the needed residue classes in a single step.
Let $w_0=\sum_a\sum_{i=1}^{d-1}u_{a,i}$ be the sum obtained by
choosing all base endpoints, where $a$ runs over the remaining
residues other than $b$. Choose one element from each new pair. In the pairs belonging to
each residue $a$, switching the first $j$ choices from $u_{a,i}$ to
$v_{a,i}$, for $0\le j\le d-1$, adds $j(a-b)$ modulo $d$.
Let $g=\gcd(d,\{a-b:a\text{ remains}\})$, taking $g=d$ if only
the base class remains. The choices for different residues are independent.
Reducing each coefficient modulo $d$ gives
\[
 \left\{\sum_a j_a(a-b)+d\mathbb Z:0\le j_a<d\right\}
 =g\mathbb Z/d\mathbb Z,
\]
where the sum is over the remaining residues other than $b$, and the
$j_a$ are integers. Thus the chosen sums represent precisely the coset
$w_0+g\mathbb Z$ modulo $d$.
They use the same number of elements, and their largest minus smallest
value is at most $d^2n$.
For each $0\le r<d$ with $g\mid r$, choose one such sum
$w_r\equiv w_0+r\pmod d$, retaining the defined value of $w_0$ for $r=0$.
Each $w_r$ has a representation using exactly one endpoint
from every new pair. All endpoints of the new pairs lie in
$B\setminus(B_1\cup E)$.
Adjoin both endpoints of every new pair to $B_1$, and increase
$\ell$ by the number of new pairs. Then
\[
 P+w_r\subseteq\ell^{\wedge}B_1
 \qquad(0\le r<d,\ g\mid r).
\]
The representations of different $w_r$ may share endpoints.
The union of the translates $P+w_r$ contains every
point of $\min P+w_0+g\mathbb Z$ in the common interval
\[
 [\min P+\max_r w_r,\ \max P+\min_r w_r].
\]
A point of this coset in the interval has the residue of some translate,
and lies between that translate's endpoints. The interval has length at least $\max P-\min P-d^2n$.
Trimming its endpoints loses less than $2g$, so the resulting
progression has common difference $g$ and diameter at least
\[
 (d^2+3)n-d^2n-2g\ge n
\]
for sufficiently large $n$.
Let $B'=B\setminus(B_1\cup E)$.
Then $|B_1|=O_c(\log n)$, $|E|=O_c(1)$ and
$B'\subseteq b+g\mathbb Z$. Since
$1\le g\le d\le\lceil4/c\rceil$, the assertion follows with
$g$ as the common difference in the statement.
\end{proof}

\begin{lemma}
\label{lem:sharp-layer-separation}
Let $U,V$ be disjoint finite sets of positive integers such that
$U\cup V$ is admissible and $U$ is nonempty. Suppose that
$U\subseteq a+q\mathbb Z$ for positive integers $a,q$, and that
some $\ell^{\wedge}V$ contains an arithmetic progression $P$ of
common difference $q$ whose diameter is at least
$\max U-\min U$. Let $r=q/\gcd(a,q)$. Then all elements of $U$
are congruent modulo $r$, and
\[
 \max(j^{\wedge}U)<\min((j+r)^{\wedge}U)
 \qquad(0\le j\le |U|-r).
\]
The inequalities are vacuous when $|U|<r$.
\end{lemma}

\begin{proof}
For each $0\le j\le|U|$, the set $P+j^{\wedge}U$ contains every
point of one residue class modulo $q$ between its extreme elements.
To prove this, start with the $j$ smallest
elements of $U$, increase the largest selected index one step at a
time to its final position, and then do the same with the remaining
indices in descending order. This gives an increasing chain of
$j$-term sums from $\min(j^{\wedge}U)$ to $\max(j^{\wedge}U)$.
Every increment is a positive multiple of $q$ and is at most
$\max U-\min U$. Consequently the progressions obtained by adding
$P$ to successive sums overlap, or at least have no missing lattice
point between them. Their union contains every point congruent to
$\min P+ja$ modulo $q$ in the interval
\[
 [\min P+\min(j^{\wedge}U),\ 
   \max P+\max(j^{\wedge}U)].
\]
For $j=0$ and $j=|U|$, the same assertion follows from the unique
sum at that cardinality.

Admissibility separates the two progressions corresponding to $j$
and $j+r$.
Since $ra\equiv0\pmod q$, the filled progressions for $j$ and $j+r$
have the same residue modulo $q$. They lie, respectively, in
$(\ell+j)^{\wedge}(U\cup V)$ and
$(\ell+j+r)^{\wedge}(U\cup V)$, because $U$ and $V$ are disjoint.
Admissibility makes these two progressions disjoint. Positivity of
$U$ makes both endpoints of the second interval larger than the
corresponding endpoints of the first. If the left endpoint of the
second interval were at most the right endpoint of the first, that point would
belong to both filled progressions, since their residues modulo $q$
agree. This contradicts admissibility, so
\[
 \max P+\max(j^{\wedge}U)
 <\min P+\min((j+r)^{\wedge}U).
\]
This implies the required strict inequality. Finally, $r\mid q$,
so the congruence condition modulo $r$ follows from that modulo $q$.
\end{proof}

\begin{corollary}
\label{cor:sharp-affine-core}
Fix $c>0$. For all sufficiently large integers $n$, suppose
$B\subseteq[n]$ has $|B|\ge c n$, and let $u\in\mathbb R$ and $v>0$ be such that
$u+vB$ is a positive integer admissible set.
There is a subset $U\subseteq u+vB$ obtained by deleting
$O_c(\log n)$ elements, and a positive integer $r$, such that
all elements of $U$ are congruent modulo $r$ and $U$ has separated
cardinality layers at distance $r$.
\end{corollary}

\begin{proof}
We transfer the smoothing decomposition to the given affine image.
The hypotheses imply $0<c\le1$. Apply
Lemma~\ref{lem:sharp-smoothing} to $B$, and let
$U=u+vB'$ and $V=u+vB_1$. The set $U$ is nonempty for sufficiently
large $n$. If $a\in U$ and $q=vd$, then
$U\subseteq a+q\mathbb Z$. The progression in
$\ell^{\wedge}B_1$ yields, after multiplication by $v$ and translation
by $\ell u$, a progression in $\ell^{\wedge}V$ of common difference
$q$ and diameter at least $vn$.
The progression has at least two terms, all integer subset sums, so
its common difference $q=vd$ is a positive integer.
Its diameter is larger than
$v(n-1)\ge\max U-\min U$, independently of $u$ and $v$.
Lemma~\ref{lem:sharp-layer-separation} with
$r=q/\gcd(a,q)$ now gives both assertions about $U$.
The deleted set is the affine image of $B_1\cup E$, so it has
$O_c(\log n)$ elements.
\end{proof}

\begin{proposition}\label{prop:almost-all}
For every $t>0$ and all sufficiently large $N$, every
admissible $A\subseteq[N]$ has a subset $U$ with
$|A\setminus U|\le t\sqrt N$ such that either $U$ is empty,
or, for some positive integer $r$, all elements of $U$ are congruent
modulo $r$ and
\[
 \max(j^{\wedge}U)<\min((j+r)^{\wedge}U)
 \qquad(0\le j\le |U|-r).
\]
\end{proposition}

\begin{proof}
We first find a reflected prefix containing almost all elements.
The assertion for $t\ge1$ follows from the case $t=1/2$, since
a deletion of at most $\frac12\sqrt N$ elements also satisfies
the bound $t\sqrt N$. We may therefore assume $0<t<1$.
If $k=|A|\le t\sqrt N$, take $U=\varnothing$. We may therefore assume
$k>t\sqrt N$.
Fix $c=t/(8\sqrt6)$. For sufficiently large $N$, there
are at least $c k/2$ indices
$k-\lfloor c k\rfloor\le i<k$. Here $A_i$ and $z_i$ have the
meanings in the proof of Proposition~\ref{prop:prefix}.
With $h=1$, equation~\eqref{eq:lifted-upper} gives
$\sum_{i=0}^{k-1}z_i\le kN$, so one such index satisfies
$z_i\le2N/c$. For $x=a_{i+1}$ and $D=x-A_i$,
Lemma~\ref{lem:chains} gives
\[
 |D|=i\ge(1-c)k,\qquad
 |\FS(D)|\le2N/c\le\frac{8}{c t^2}i^2.
\]

We next place almost all of this prefix in a short arithmetic
progression. Apply Corollary~\ref{cor:integer-CDK} to
$D\subseteq\N$ with $c_0=8/(c t^2)$, and denote its positive
integer step by $v$.
There are $D_1\subseteq D$ and positive constants $c_2,c_3$,
depending only on $t$, such that
\[
 |D\setminus D_1|\le c_2,\qquad D_1\subseteq v\mathbb Z_{>0},
 \qquad\sum_{d\in D_1}d/v\le c_3i^2.
\]
We may take $c_3\ge1$. Define
\[
 A'=\{x-d:d\in D_1,\ d/v\le c_3k/c\}\subseteq A_i.
\]
The map $a\mapsto x-a$ is a bijection from $A_i$ to $D$, and
\[
 \abs{\{d\in D_1:d/v>c_3k/c\}}
 \le \frac{\sum_{d\in D_1}d/v}{c_3k/c}
 \le \frac{ci^2}{k}\le ck.
\]
The deleted elements form three disjoint sets:
$A\setminus A_i$, $x-(D\setminus D_1)$, and
$\{x-d:d\in D_1,\ d/v>c_3k/c\}$.
Their sizes are at most $ck$, $c_2$, and $ck$, respectively.
In particular, $x=a_{i+1}\in A\setminus A_i$ is included in the
first set. Hence
\[
 |A'|\ge(1-2c)k-c_2,
 \qquad
 A'\subseteq\{x-v j:1\le j\le n\},\quad
 n=\lceil c_3k/c\rceil.
\]
In particular $n=O_t(k)$ and $|A'|\ge k/2$ for large $N$.

To apply Corollary~\ref{cor:sharp-affine-core}, let
$B=\{n+1-(x-a)/v:a\in A'\}$.
The preceding truncation ensures that each $(x-a)/v$ is an
integer in $[n]$, so $B\subseteq[n]$ and $A'=x-v(n+1)+vB$.
The set $A'$ is positive and admissible because it is a subset of $A$.
For sufficiently large $N$, we have $n\le2c_3k/c$ and
$|B|=|A'|\ge k/2$, giving $|B|/n\ge c/(4c_3)$.
Corollary~\ref{cor:sharp-affine-core}, with density parameter
$c/(4c_3)$ and affine parameters $x-v(n+1),v$, gives the required
modular separated subset $U$, after deleting
$O_t(\log n)=O_t(\log k)$ further elements.
Using $k\le\sqrt{6N}$ from Lemma~\ref{lem:size}, we obtain
\[
 |A\setminus U|
 \le2c k+c_2+O_t(\log k)
 \le\frac{t}{4}\sqrt N+O_t(\log N)
 \le t\sqrt N
\]
for sufficiently large $N$. All thresholds depend only on the fixed
$t$, proving the assertion uniformly over $A$.
\end{proof}

\begin{proof}[Proof of Theorem~\ref{thm:sharp}]
We first count all admissible subsets of $[N]$ by their two parts.
Fix $0<t<1$, and select $U$ in
Proposition~\ref{prop:almost-all} by a fixed rule. Let
$V=A\setminus U$. The ordered pair $(U,V)$ determines $A$ by union.
The set $V$ is admissible and has size at most $t\sqrt N$.
Lemma~\ref{lem:small-admissible-cost} bounds its number of choices by
$\exp([30t\log(e/t)+o_t(1)]\sqrt N)$.
Lemma~\ref{lem:uniform-modular-cost} bounds the number of possible $U$
by $N^2p(N)+N+1$.
Counting the two parts independently and using the Hardy--Ramanujan
formula gives
\[
 \limsup_{N\to\infty}\frac{\log Q_1(N)}{\sqrt N}
 \le\pi\sqrt{\frac23}+30t\log(e/t).
\]
Taking $t\downarrow0$ after this limit gives the sharp
upper bound, and $b(N)\le Q_1(N)$ gives it for $b(N)$ as well.

We obtain the matching lower bound from stratified sets.
Theorem~\ref{thm:nicolas-count} gives exactly $p(N)$ stratified sets with
maximum $N$. They are admissible, so $Q_1(N)\ge b(N)\ge p(N)$.
The Hardy--Ramanujan formula \cite{HR1918} implies
$\log p(N)=\pi\sqrt{2N/3}+O(\log N)$.
Combining the lower and upper bounds gives the asymptotic for
$\log b(N)$ and $\log Q_1(N)$. Finally,
$b(N)\le a(N)=Q_1(N)-1\le Q_1(N)$ gives the same asymptotic
for $\log a(N)$.
\end{proof}

\begin{corollary}\label{cor:cumulative}
For every positive integer $N$,
\[
 a(N)=\sum_{n=1}^{N}b(n)\ge\sum_{n=1}^{N}p(n),
 \qquad
 \sum_{n=1}^{N}p(n)\sim\frac{\sqrt{6N}}{\pi}\,p(N)
 \quad(N\to\infty).
\]
\end{corollary}

\begin{proof}
We first count the sets according to their largest element.
The classes with distinct maxima are disjoint, so
$a(N)=\sum_{n=1}^{N}b(n)$. By Theorem~\ref{thm:nicolas-count},
$b(n)\ge p(n)$ for each $n$. Summing gives
$a(N)\ge\sum_{n=1}^{N}p(n)$.

We estimate the cumulative partition count directly from
\eqref{eq:HR}. For real $x>0$, define
$f(x)=\exp(\pi\sqrt{2x/3})/(2\pi\sqrt{2x})$.
Differentiation gives
$f'(x)=f(x)(\pi/\sqrt{6x}-1/(2x))$.
For $N-1\le x\le N$, we have $x/N\to1$ and
$\sqrt x-\sqrt N=O(N^{-1/2})$ uniformly. The derivative formula
and \eqref{eq:HR} therefore give $f'(x)\sim p(N)$ uniformly on
this interval. By the mean value theorem,
$f(N)-f(N-1)\sim p(N)$. Since $f(N)$ is eventually strictly
increasing and tends to infinity, the Stolz--Ces\`aro theorem yields
$\sum_{n=1}^{N}p(n)\sim f(N)\sim\sqrt{6N}\,p(N)/\pi$.
\end{proof}

Theorem~\ref{thm:sharp} determines the logarithmic growth of $a(N)$,
whereas Corollary~\ref{cor:cumulative} supplies an explicit lower
bound. In the proof of Theorem~\ref{thm:sharp}, the limit in $N$
is taken with $t$ fixed, and then $t\downarrow0$.

\section{Infinite sets}\label{sec:infinite}

\begin{proposition}\label{prop:finite-deletion}
For $A\subseteq\N$, one has
\[
 \q(A)<\infty
 \quad\Longleftrightarrow\quad
 A\setminus F\text{ is admissible for some finite }F\subseteq A.
\]
Moreover, if $F$ is such a set, then $\q(A)\le2^{|F|}$.
\end{proposition}

\begin{proof}
To obtain a finite deletion, suppose $\q(A)\le h$ for a positive
integer $h$, but $A\setminus F$ is not admissible for any finite
$F\subseteq A$. Apply this assumption successively, each time taking
$F$ to be the union of the previously selected sets. This gives finite
$U_i,V_i\subseteq A$, $1\le i\le h$, such that
the sets $U_i\cup V_i$ are pairwise disjoint,
$\sum_{u\in U_i}u=\sum_{v\in V_i}v$, and $|U_i|<|V_i|$.
For $0\le j\le h$, let
\[
 W_j=\left(\bigcup_{i=1}^{j}V_i\right)
       \cup\left(\bigcup_{i=j+1}^{h}U_i\right).
\]
Each $W_j$ has sum $\sum_{i=1}^{h}\sum_{u\in U_i}u$, and
$|W_j|-|W_{j-1}|=|V_j|-|U_j|>0$ for $1\le j\le h$.
Thus $\q(A)\ge h+1$, a contradiction.

To prove the converse, suppose $F\subseteq A$ is finite and
$A\setminus F$ is admissible. Fix $s\in\Z$ and $T\subseteq F$.
Every finite $U\subseteq A$ with sum $s$ and $U\cap F=T$ satisfies
\[
 \sum_{u\in U\setminus F}u=s-\sum_{t\in T}t.
\]
Admissibility of $A\setminus F$ determines $|U\setminus F|$, and hence
$|U|=|U\setminus F|+|T|$. The $2^{|F|}$ choices of $T$ therefore give
$|L_A(s)|\le2^{|F|}$ and $\q(A)\le2^{|F|}$.
\end{proof}

\begin{proof}[Proof of Theorem~\ref{thm:infinite}]
By Proposition~\ref{prop:finite-deletion}, it suffices to consider
admissible $A$: deleting a finite set preserves infinitude and
changes the counting function by $O(1)$.

We first obtain reflected prefixes with quadratically many subset sums.
Let $A=\{a_1<a_2<\cdots\}$ and suppose, for a contradiction, that
$A(N)\ge t\sqrt N$ for some $t>0$ and an unbounded
sequence of positive integers $N$. Each finite prefix
is admissible. For all sufficiently large selected $N$,
Proposition~\ref{prop:prefix}, applied to $A\cap[N]$ with $h=1$,
gives an index
$\lceil A(N)/2\rceil\le i<A(N)$. With $x=a_{i+1}$, the set
$$
 D=\{x-a_j:1\le j\le i\}
$$
satisfies
\[
 \abs{\FS(D)}\le3N
 \le\frac{12}{t^2}i^2.
\]
Here $i\to\infty$ along the selected sequence.

We bound the integer step using a fixed initial segment of $A$.
Apply Corollary~\ref{cor:integer-CDK} to $D\subseteq\N$
with $c_0=12/t^2$.
Choose the exceptional-set bound $m_0$ to be an integer. The constants
$m_0,c_1$ depend only on $t$, whereas the decomposition
$D=D_1\sqcup D_2$ and the step $r$ may depend on the selected $N$.
For large $N$ we have $i\ge m_0+2$. At least two of the $m_0+2$ differences
$x-a_1,\ldots,x-a_{m_0+2}$ belong to $D_1$. Choose retained indices $j<\ell\le m_0+2$.
Since $r\mid(x-a_j)$ and $r\mid(x-a_\ell)$, one has
$r\mid(a_\ell-a_j)$. As $a_\ell-a_j>0$, it follows that
\begin{equation}\label{eq:step-bound}
 r\le a_\ell-a_j\le a_{m_0+2}-a_1.
\end{equation}
The bound $a_{m_0+2}-a_1$ depends on the fixed initial segment of $A$ and is
independent of the selected $N$.

We bound the sum of the retained differences in the original scale.
Equation~\eqref{eq:CDK} gives
\begin{equation}\label{eq:unnormalized}
 \sum_{d\in D_1}d\le rc_1i^2.
\end{equation}
The prefix $\{a_1,\ldots,a_{i+1}\}\subseteq[x]$ is admissible.
Lemma~\ref{lem:size}, with $h=1$, therefore gives $i^2\le6x$, and
\eqref{eq:unnormalized} becomes
\[
 \sum_{d\in D_1}d\le6rc_1x.
\]

We now bound the number of elements below $x/2$ uniformly in $x$.
Every element of $A\cap[1,x/2]$ belongs to $\{a_1,\ldots,a_i\}$
and contributes a difference $x-a_j\ge x/2$. At most $m_0$ of these
differences lie in $D_2$. Thus
$$
 \bigl(A(x/2)-m_0\bigr)\frac x2
 \le\sum_{d\in D_1}d\le6rc_1x,
$$
which, together with \eqref{eq:step-bound}, gives
$A(x/2)\le m_0+12rc_1\le m_0+12c_1(a_{m_0+2}-a_1)$.
But $i\to\infty$ implies $x=a_{i+1}\to\infty$, and the infinitude of
$A$ implies $A(x/2)\to\infty$. This contradiction proves
$A(N)=o(\sqrt N)$ for integer $N$. The real-variable statement follows
from $A(x)=A(\lfloor x\rfloor)$.
\end{proof}

\begin{corollary}\label{cor:growth}
If $A=\{a_1<a_2<\cdots\}\subseteq\N$ is infinite and $\q(A)<\infty$,
then
$$
 \frac{a_n}{n^2}\longrightarrow\infty,
 \qquad
 \limsup_{n\to\infty}\frac{a_{n+1}-a_n}{n}=\infty.
$$
In particular, no infinite admissible set has $a_{n+1}-a_n=O(n)$.
\end{corollary}

\begin{proof}
To prove $a_n/n^2\to\infty$, observe that $A(a_n)=n$.
Theorem~\ref{thm:infinite} gives $n/\sqrt{a_n}\to0$.

We exclude a bounded ratio of successive gaps to their indices.
If the second assertion failed,
the nonnegative ratios $(a_{n+1}-a_n)/n$ would be bounded for all
sufficiently large $n$. Summing $a_{n+1}-a_n=O(n)$ would give
$a_n=O(n^2)$, contradicting $a_n/n^2\to\infty$.
\end{proof}

\begin{remark}
Theorem~\ref{thm:infinite} is a statement for each fixed infinite set.
The bound on the step in \eqref{eq:step-bound} uses an initial segment
of that set, and no uniform rate for the little-$o$ term is asserted.
Corollary~\ref{cor:growth} excludes linear gap bounds.
\end{remark}

\section{Exact enumeration of Nicolas's odd family}\label{sec:odd}

For the proof of Theorem~\ref{thm:odd}, we temporarily allow negative
odd elements while retaining condition~\eqref{eq:odd-property} and a
positive maximum. For brevity, say that such a set has the
\emph{two-layer separation property}. A set of positive odd integers
with this property is admissible: two equal sums must use cardinalities
of the same parity, and repeated use of \eqref{eq:odd-property}
separates any two distinct layers of that parity.
For example, $\{3,5,7,9\}$ has this property, although it is not
stratified because $9>3+5$.

\begin{lemma}\label{lem:odd-transform}
Fix a positive odd integer $N$. The map
\[
 A\longmapsto T=
 \left\{\frac{a+N}{2}+1:a\in A\right\}\cup\{N+2\}
\]
is a bijection between the finite sets of odd integers with maximum $N$ and
the two-layer separation property, and the stratified sets whose two
largest elements are $N+1,N+2$.
\end{lemma}

\begin{proof}
We first verify that the image consists of distinct positive integers.
If $\abs A=1$, the map sends $\{N\}$ to
$\{N+1,N+2\}$. Suppose $\abs A\ge2$. The inequality in
\eqref{eq:odd-property} for $j=0$ says that the two smallest elements
of $A$ have positive sum. Since the second smallest is at most $N$,
we have $\min A>-N$. All the elements of $T$ are therefore positive
integers, and its two largest elements are $N+1,N+2$.

We compare the interior cardinality layers of the image.
For $0\le j\le\abs A-2$, the largest $(j+1)$-term sum of $T$ includes
$N+2$, while the smallest $(j+2)$-term sum excludes it. Consequently
\begin{align*}
 \max((j+1)^{\wedge}T)
 &=N+2+\frac{\max(j^{\wedge}A)+jN}{2}+j,\\
 \min((j+2)^{\wedge}T)
 &=\frac{\min((j+2)^{\wedge}A)+(j+2)N}{2}+j+2.
\end{align*}
Subtracting gives
\begin{equation}\label{eq:odd-layer-difference}
 \min((j+2)^{\wedge}T)-\max((j+1)^{\wedge}T)
 =\frac{\min((j+2)^{\wedge}A)-\max(j^{\wedge}A)}2.
\end{equation}
By \eqref{eq:odd-property}, the right-hand side is positive.
This gives all the interior adjacent-layer inequalities for $T$.
The two endpoint inequalities follow from positivity, so $T$ is
stratified.

We recover the original set by reversing the affine transformation.
From a stratified $T$ with consecutive largest elements
$N+1,N+2$, recover
$$
 A=\{2t-(N+2):t\in T\setminus\{N+2\}\}.
$$
The set $A$ has odd elements and maximum $N$. Since $T$ is
stratified, the left-hand side of \eqref{eq:odd-layer-difference}
is positive for $0\le j\le\abs A-2$. Equation~\eqref{eq:odd-layer-difference}
therefore gives \eqref{eq:odd-property} for $A$. The two maps
are mutual inverses.
\end{proof}

\begin{lemma}\label{lem:consecutive-top}
For every integer $M\ge3$, the number of stratified sets whose two
largest elements are $M-1,M$ is
\[
 p(M)-p(M-2)-1.
\]
\end{lemma}

\begin{proof}
We count the complementary condition on the largest two elements.
Consider a stratified set of size $k\ge3$ with maximum $M$ whose
second largest element is at most $M-2$. Decrease its largest element
by two and every other element by one. The resulting set still has
$k$ distinct positive elements, and its maximum is $M-2$.
For positivity, observe that a stratified set of size at least three
cannot have smallest element $1$: its two smallest elements would
then sum to at most its largest, contradicting stratification.
The original top gap of at least two becomes a gap of at least one.

We check that the transformation and its inverse preserve stratification.
For each interior comparison, with $1\le j\le k-2$, the largest
$j$-term sum includes the maximum and decreases by $j+1$. The smallest
$(j+1)$-term sum excludes the maximum and decreases by the same
amount. All these inequalities are therefore preserved. The endpoint
inequalities follow from positivity. The inverse map increases the
maximum by two and all other elements by one. It preserves the
interior inequalities and creates a top gap of at least two. This is a
bijection with the stratified sets of maximum $M-2$ having at least
three elements, and it preserves cardinality.

We include the two small sizes in the count.
By Theorem~\ref{thm:nicolas-count}, there are
$p(M)-M$ stratified sets of maximum $M$ having at least three
elements: the singleton and the $M-1$ two-element sets account
for the other $M$ sets. The corresponding number at maximum $M-2$
is $p(M-2)-(M-2)$, including $M=3$ since $p(1)-1=0$.
Exactly one two-element set has largest elements $M-1,M$.
Thus the required count is
\[
 1+[p(M)-M]-[p(M-2)-(M-2)]
 =p(M)-p(M-2)-1.
\]
\end{proof}

\begin{proof}[Proof of Theorem~\ref{thm:odd}]
We first count the family without the positivity restriction.
Lemmas~\ref{lem:odd-transform} and~\ref{lem:consecutive-top} show that
there are $p(N+2)-p(N)-1$ sets with maximum $N$ when negative odd
elements are allowed.

We classify the sets that contain a negative element.
There is at most one such element, because the two smallest
elements have positive sum. Let this element be $-r$, where
$r>0$ is odd, and let $B=A\setminus\{-r\}$. Every element of $B$
exceeds $r$. For $0\le j\le\abs B-1$, the relevant extrema are
$$
 \max(j^{\wedge}A)=\max(j^{\wedge}B),\qquad
 \min((j+2)^{\wedge}A)=-r+\min((j+1)^{\wedge}B).
$$
The two-layer inequalities for $A$ therefore become
$$
 \max(j^{\wedge}B)<\min((j+1)^{\wedge}B)-r.
$$
Subtracting $r$ from each element converts these precisely to the
adjacent-layer inequalities for $B-r$. Hence $T=\{(b-r)/2:b\in B\}$ is a positive stratified set, with maximum
$m=(N-r)/2$. The possible values are exactly
$1\le m\le(N-1)/2$. Conversely, every stratified set $T$ with
maximum $m$ in this range gives the unique set
$$
 r=N-2m,\qquad A=\{-r\}\cup\{2t+r:t\in T\}.
$$
The negative element determines $m$, so different values of $m$ give
disjoint cases. By Theorem~\ref{thm:nicolas-count}, there are $p(m)$
choices for $T$ with maximum $m$, giving a total of
$\sum_{m=1}^{(N-1)/2}p(m)$. Subtracting this number proves
\eqref{eq:odd-formula}. For $N=1$ the correction sum is empty and the
formula gives $\bo(1)=p(3)-p(1)-1=1$, as required.

We first estimate the difference of consecutive partition numbers.
For $n\ge2$, let $d_n=p(n)-p(n-1)$, the number of partitions of $n$
having no part equal to $1$. Increasing the largest part by one
gives an injection into the corresponding partitions of $n+1$:
the largest part of the image is unique, and decreasing it by one
recovers the original partition. Thus $d_n$ is nondecreasing.
For fixed $s>0$ and sufficiently large $n$, let
$\ell=\lfloor s\sqrt n\rfloor$. Then
\[
 \frac{p(n)-p(n-\ell)}{\ell}\le d_n
 \le\frac{p(n+\ell)-p(n)}{\ell}.
\]
Equation~\eqref{eq:HR} gives
$p(n\pm\ell)/p(n)\to\exp(\pm\pi s/\sqrt6)$, so
\[
 \frac{1-e^{-\pi s/\sqrt6}}s
 \le\liminf_{n\to\infty}\frac{\sqrt n\,d_n}{p(n)}
 \le\limsup_{n\to\infty}\frac{\sqrt n\,d_n}{p(n)}
 \le\frac{e^{\pi s/\sqrt6}-1}s.
\]
Letting $s\downarrow0$ yields
$d_n\sim\pi p(n)/\sqrt{6n}$, and hence
\[
 p(N+2)-p(N)=d_{N+2}+d_{N+1}
 \sim\pi\sqrt{\frac{2}{3N}}\,p(N).
\]
By \eqref{eq:HR},
$1+Np(\lfloor N/2\rfloor)=o(p(N)/\sqrt N)$.
The correction term in \eqref{eq:odd-formula} is at most this
quantity, proving the asserted asymptotic for $\bo(N)$.

We obtain the cumulative asymptotic by summing
\eqref{eq:odd-formula}. For every $m\ge1$,
\[
 \sum_{j=1}^{m}\bo(2j-1)
 =p(2m+1)-m-1-\sum_{j=1}^{m-1}(m-j)p(j).
\]
The subtracted terms are $O(m^2p(m))=o(p(2m-1))$ by
\eqref{eq:HR}. Equation~\eqref{eq:HR} also gives
$p(2m+1)\sim p(2m)\sim p(2m-1)$.
This proves the cumulative asymptotic for both odd and even $N$.
\end{proof}

\begin{corollary}\label{cor:odd-excess}
We have
\[
 \begin{aligned}
  \liminf_{\substack{N\to\infty\\N\text{ odd}}}
  \frac{\sqrt N\,[b(N)-p(N)]}{p(N)}
  &\ge\pi\sqrt{\frac23},\\
  \liminf_{N\to\infty}
  \frac{a(N)-\sum_{n=1}^{N}p(n)}{p(N)}
  &\ge1.
 \end{aligned}
\]
\end{corollary}

\begin{proof}
We first count the stratified sets among those counted by $\bo(N)$.
For odd $N$, the map $T\mapsto A=2T-1$ is a bijection between
stratified sets of maximum $(N+1)/2$ and stratified sets of odd
positive integers with maximum $N$. Indeed, for
$0\le j<\abs T$,
\[
 \min((j+1)^{\wedge}A)-\max(j^{\wedge}A)
 =2\bigl(\min((j+1)^{\wedge}T)-\max(j^{\wedge}T)\bigr)-1.
\]
The quantity in parentheses is an integer, so the adjacent-layer
inequalities for $A$ and $T$ are equivalent.
Every stratified set satisfies \eqref{eq:odd-property}.
By Theorem~\ref{thm:nicolas-count}, the intersection has size $p((N+1)/2)$.
Counting the union of the stratified family and the family counted
by $\bo(N)$ gives
\[
 b(N)-p(N)\ge\bo(N)-p((N+1)/2).
\]
Since $p((N+1)/2)=o(p(N)/\sqrt N)$ by \eqref{eq:HR},
Theorem~\ref{thm:odd} proves the first assertion.

We sum the additional sets over their distinct maxima.
Using $b(n)\ge p(n)$ also for even $n$, we obtain
\[
 a(N)-\sum_{n=1}^{N}p(n)
 \ge
 \sum_{\substack{1\le n\le N\\n\text{ odd}}}\bo(n)
 -\sum_{m=1}^{\lceil N/2\rceil}p(m).
\]
By \eqref{eq:HR},
$\sum_{m=1}^{\lceil N/2\rceil}p(m)\le Np(\lceil N/2\rceil)=o(p(N))$.
The cumulative asymptotic in Theorem~\ref{thm:odd} proves the
second assertion.
\end{proof}


\begin{thebibliography}{99}

\bibitem{Campos2020}
M. Campos,
\emph{On the number of sets with a given doubling constant},
Israel J. Math. \textbf{236} (2020), 711--726.
\href{https://arxiv.org/abs/1811.05793v3}{arXiv:1811.05793v3}.

\bibitem{CDK2026}
R. Carpenter, C. Defant and N. Kravitz,
\emph{Sets with few subset sums},
preprint (2026),
\href{https://arxiv.org/abs/2605.05498v2}{arXiv:2605.05498v2},
revised 24 August 2026.

\bibitem{DF1995}
J.-M. Deshouillers and G. A. Freiman,
\emph{On an additive problem of Erd\H{o}s and Straus, 1},
Israel J. Math. \textbf{92} (1995), 33--43.

\bibitem{DF1999}
J.-M. Deshouillers and G. A. Freiman,
\emph{On an additive problem of Erd\H{o}s and Straus, 2},
Ast\'erisque \textbf{258} (1999), 141--148.

\bibitem{Erdos1962}
P. Erd\H{o}s,
\emph{Sz\'amelm\'eleti megjegyz\'esek, III}
[Some remarks on number theory, III],
Mat. Lapok \textbf{13} (1962), 28--38 (Hungarian).

\bibitem{ENS1991}
P. Erd\H{o}s, J.-L. Nicolas and A. S\'ark\"ozy,
\emph{Sommes de sous-ensembles},
J. Th\'eor. Nombres Bordeaux \textbf{3} (1991), no.~1, 55--72.

\bibitem{GM2016}
B. Green and R. Morris,
\emph{Counting sets with small sumset and applications},
Combinatorica \textbf{36} (2016), 129--159.

\bibitem{HR1918}
G. H. Hardy and S. Ramanujan,
\emph{Asymptotic formulae in combinatory analysis},
Proc. London Math. Soc. (2) \textbf{17} (1918), 75--115.

\bibitem{LMS2025}
D. Liu, L. Mattos and T. Szab\'o,
\emph{On the number of sets with small sumset},
Israel J. Math., to appear.
\href{https://arxiv.org/abs/2407.04492v2}{arXiv:2407.04492v2} (2025).

\bibitem{Nicolas1999}
J.-L. Nicolas,
\emph{Stratified sets},
Ast\'erisque \textbf{258} (1999), 205--215.

\bibitem{Straus1966}
E. G. Straus,
\emph{On a problem in combinatorial number theory},
J. Math. Sci. \textbf{1} (1966), 77--80.

\end{thebibliography}
\end{document}